\documentclass[journal,twoside,web]{ieeecolor}
\usepackage{generic}
\usepackage{cite}
\usepackage{amsmath,amssymb,amsfonts}
\usepackage{algcompatible}
\usepackage{algorithm}
\usepackage{cite}
\usepackage[pdftex]{graphicx}
\usepackage[caption=false,font=footnotesize]{subfig}
\usepackage[hidelinks]{hyperref}
\usepackage{bm}

\newtheorem{theorem}{Theorem}
\newtheorem{lemma}{Lemma}

\newtheorem{corollary}{Corollary}

\DeclareMathOperator*{\argmin}{arg\,min}
\DeclareMathOperator{\E}{\mathbb{E}}

\DeclareMathOperator{\diag}{diag}

\newcommand{\RR}{\mathbb{R}}

\newcommand{\qedproof}{\hfill\blacksquare}

\newcommand{\Reg}{\texttt{R}}
\newcommand{\XX}{\mathcal{X}}
\newcommand{\xx}{\mathbf{x}}
\newcommand{\yy}{\mathbf{y}}
\newcommand{\uu}{\mathbf{u}}
\newcommand{\xst}{\mathbf{x}^\star_t}
\newcommand{\xt}{\mathbf{x}_t}
\newcommand{\hxt}{\mathbf{\hat{x}}_t}
\newcommand{\pt}{\mathbf{p}_t}
\newcommand{\tpt}{\tilde{\mathbf{p}}_t}
\renewcommand{\tt}{\bm{\theta}_t}
\newcommand{\mtt}{\left\langle \bm{\theta} \right\rangle_{t-1}}

\newcommand{\oo}{\mathbf{1}}

\def\BibTeX{{\rm B\kern-.05em{\sc i\kern-.025em b}\kern-.08em
    T\kern-.1667em\lower.7ex\hbox{E}\kern-.125emX}}

\begin{document}
\title{Online Convex Optimization\\ with Binary Constraints}

\author{Antoine Lesage-Landry, \IEEEmembership{Member, IEEE}, Joshua A. Taylor, \IEEEmembership{Member, IEEE},\\ and Duncan S. Callaway, \IEEEmembership{Member, IEEE} 
\thanks{This work was funded by the Natural Sciences and Engineering Research Council of Canada, by the Institute for Data Valorization (IVADO), by the National Science Foundation, award 1351900, and by the Advanced Research Projects Agency-Energy, award DE-AR0001061.}
\thanks{A. Lesage-Landry is with the Department of Electrical Engineering, Polytechnique Montréal, Montréal, QC, Canada, H3T 1J4. e-mail: \texttt{antoine.lesage-landry@polymtl.ca}}%
\thanks{J.A. Taylor is with The Edward S. Rogers Sr. Department of Electrical \& Computer Engineering, University of Toronto, Toronto, Ontario, Canada, M5S 3G4. e-mail: \texttt{josh.taylor@utoronto.ca}}%
\thanks{D.S. Callaway is with the Energy \& Resources Group, University of California, Berkeley, CA, USA 94720. e-mail: \texttt{dcal@berkeley.edu}}%
}

\maketitle

\begin{abstract}
We consider online optimization with binary decision variables and convex loss functions. We design a new algorithm, binary online gradient descent (\texttt{bOGD}) and bound its expected dynamic regret. We provide a regret bound that holds for any time horizon and a specialized bound for finite time horizons. First, we present the regret as the sum of the relaxed, continuous round optimum tracking error and the rounding error of our update in which the former asymptomatically decreases with time under certain conditions. Then, we derive a finite-time bound that is sublinear in time and linear in the cumulative variation of the relaxed, continuous round optima. We apply \texttt{bOGD} to demand response with thermostatically controlled loads, in which binary constraints model discrete on/off settings. We also model uncertainty and varying load availability, which depend on temperature deadbands, lockout of cooling units and manual overrides. We test the performance of \texttt{bOGD} in several simulations based on demand response. The simulations corroborate that the use of randomization in \texttt{bOGD} does not significantly degrade performance while making the problem more tractable.
\end{abstract}


\section{Introduction}
\label{sec:intro}

\IEEEPARstart{I}{n} online convex optimization (OCO), one makes a decision in each time step to minimize a time-varying loss function~\cite{hazan2016introduction,shalev2012online}. The loss function is only observed by the decision maker following the round. These observations are then used to update the decision. The simplicity and computational efficiency of OCO make it well-suited to online applications like demand response in electric power systems~\cite{kim2016online,lesage2018setpoint}.

Demand response with thermostatically controlled loads (TCLs), e.g., residential households~\cite{callaway2009tapping,mathieu2012state} or commercial buildings~\cite{hao2014ancillary}, can provide fast timescale services such as frequency regulation and load following~\cite{callaway2010achieving}. These services can economically improve the operation of power systems~\cite{taylor2016power}. 

In this work, we formulate an algorithm for online convex optimization with binary constraints, which we refer to as the binary online gradient descent (\texttt{bOGD}). The challenge is to design an algorithm: (i) in line with OCO's simple and computationally efficient updates so that binary decisions can be provided in real-time, and (ii) which has theoretically guaranteed performance. The guarantee we are after takes the form of a sublinear regret in the problem's time horizon.
We use \texttt{bOGD} for setpoint tracking with TCLs with discrete on/off settings and lockout constraints, as well as the standard deadband temperature constraint. The lockout constraint ensures that, after being turned off, a 5-minute period must pass before a cooling unit can be turned on again. Lockout is a significant physical limitation of TCLs for fast demand response applications, and is rarely accounted for in the literature. 

Several decision-making and resource allocation problems have binary constraints, e.g., scheduling~\cite{hillier1995introduction} or turning off or on a residential load's cooling unit for demand response~\cite{callaway2009tapping,mathieu2012state}. In this work, we focus on functions which, without the binary constraint, are convex over the set $[0,1]^n$, $n \in \mathbb{N}$. In our work, we show that \texttt{bOGD}, an OCO algorithm with an additional randomization step, provides sequential binary decisions with a provable finite-time performance guarantee~\cite{shahrampour2017distributed}. We show that, for a finite time horizon, \texttt{bOGD} has an expected dynamic regret bound that is sublinear in the number of rounds and linear in the cumulative variation between relaxed, continuous round optima. If the cumulative variation is sublinear, then the average expected regret decreases as the time horizon increases. Therefore, for a finite time horizon, the expected decisions provided by \texttt{bOGD}, on average, approach the optimal ones because the average error approaches zero.

\paragraph*{Related work}
\texttt{bOGD} is based on OCO~\cite{hazan2016introduction,shalev2012online}. In online optimization, the closest work to ours is~\cite{hazan2012online}. In this work, the authors consider binary decision variables for submodular loss functions and provide a static regret analysis. Several extensions of this work have been proposed, e.g.,~\cite{jegelka2011online,cardoso2019differentially}. Similarly to convexity in continuous optimization, in discrete/combinatorial optimization, submodularity usually implies that a problem can be solved efficiently~\cite{bach2019submodular}. Submodular minimization problems can be solved efficiently using either combinatorial algorithms~\cite{iwata2001combinatorial} or convex optimization algorithms via the Lov\'asz extension, a continuous and convex extension of the submodular function~\cite{bach2019submodular,lovasz1983submodular}. In our work, we minimize convex functions over the binary set $\left\{0,1\right\}^n$, where convexity is over the continuous set $\left[0,1\right]^n$. Convexity over the continuous set does not imply submodularity over the binary set, and nor does submodularity imply convexity. Consider, for example, a twice differentiable function $f:\left[0,1\right]^n\mapsto\RR$. Then, $f$ is submodular if all of its mixed second-order derivatives are nonpositive~\cite{bach2019submodular}, whereas $f$ is convex over $\left[0,1\right]^n$ if the Hessian is positive semidefinite. While both conditions can be met simultaneously, neither of them implies the other. The loss function of this work's motivating application, setpoint tracking with TCLs, is convex but not submodular, and thus falls outside the scope of~\cite{hazan2012online}. We also provide a different performance guarantee than~\cite{hazan2012online} by showing that \texttt{bOGD} has a bounded finite-time dynamic regret bound.

Randomized expert and bandit-based algorithms have also been developed for online problems with binary decision variables~\cite{bubeck2011introduction,hazan2012online}. In the general case, these algorithms are computationally inefficient due to the large number of possible experts or decisions, i.e., $2^n$ possible options must be considered, where $n \in \mathbb{N}$ is the number of experts or decisions~\cite{hazan2012online}. Several efficient algorithms for binary decisions have been proposed in online linear optimization~\cite{audibert2014regret} (and the references therein) under restrictive assumptions, e.g., a constant total number of 1-decisions being selected and linear loss functions. 

Binary constraints were considered in the context of time-varying optimization of power systems in~\cite{zhou2019online,bernstein2019real,bernstein2016real}. In~\cite{zhou2019online}, a similar randomization step is used to accommodate loads with discrete settings. The performance of the randomization step is, however, only characterized through an asymptotic bound on the dual variables of the problem. Reference~\cite{bernstein2019real} also models discrete decisions and lockout using the approach of~\cite{bernstein2016real}. In~\cite{bernstein2016real}, the authors extend the error diffusion algorithm. In each round, the decision is given by the projection onto the discrete set of the cumulative difference between the continuous (relaxed) and the discrete decisions. An example of demand response with heaters that includes binary decisions and lockout constraints is provided. The performance analysis of~\cite{bernstein2019real} and~\cite{bernstein2016real} ensures that the cumulative difference between the continuous and discrete decisions is bounded, but it does not characterize the difference between the round optimum and the discrete decision made in each round.

Our specific contributions are:
\begin{itemize}
    \item We handle binary constraints in OCO using a randomization step. To the best of our knowledge, binary decision variables have not previously been integrated into OCO. 
    \item We obtain expected dynamic regret bounds for the resulting algorithm, \texttt{bOGD}. We show that the regret is bounded from above by the sum of the relaxed optimum tracking error, which vanishes in time when standard assumptions are satisfied, and the rounding error. For finite time horizons, we provide a bound that is sublinear in time and linear in the cumulative variation of relaxed round optima.
    \item We use \texttt{bOGD} for demand response of TCLs with binary on/off settings. We also account for the unavailability of loads for demand response caused by temperature deadbands, lockout constraints, and manual override.
\end{itemize}

\section{OCO with binary constraints}
\label{sec:oco_bin}
We consider an OCO setting in which a decision maker sequentially minimizes a loss function with binary decisions. The loss function is specific to the current time instance or round.
We denote the round by $t$ and the time horizon by $\tau$. Let $\hxt \in \left\{0,1\right\}^n$ and $\xx_t \in \left[0,1\right]^n$ be, respectively, the discrete decision and the relaxed, continuous decision variables at round $t$. Let $f_t:\left[0,1\right]^n\mapsto \RR$, $n\in \mathbb{N}$, be the differentiable loss function at round $t$ for $t=1,2,\ldots, \tau$. We assume that $f_t$ is convex and bounded over $\left[0,1\right]^n$. The convexity of $f_t$ and compactness of the relaxed domain imply that $f_t$ is Lipschitz continuous with respect to a norm $\|\cdot\|$~\cite{hiriart2013convex} for all $\xx \in \left[0,1\right]^n$ and $t=1,2,\ldots,\tau$. Let $L_1, L_2 \in (0,+\infty)$ be the Lipschitz continuity moduli with respect to the $\ell_1$- and $\ell_2$-norm, respectively. Thus, $\left\| \nabla f_t \left(\xx \right)\right\|_1 \leq L_1$ and $\left\| \nabla f_t \left(\xx \right)\right\|_2 \leq L_2$ for all $t$ and $\xx \in \left[ 0 , 1 \right]^n$ and all $t$. 

The binary-constrained OCO problem is as follows. In each round $t$, the decision maker must solve:
\begin{equation}
\min_{\hxt \in \left\{0,1 \right\}^n} f_t(\hxt), \label{eq:nonconvex}
\end{equation} 
where $f_t$ is only known after round $t$. The performance of OCO algorithms are evaluated in terms of the regret~\cite{zinkevich2003online,shalev2012online,hazan2016introduction}. 
In this work, we use the dynamic regret~\cite{zinkevich2003online,hall2015online,shahrampour2017distributed}, which represents the cumulative loss difference between the actual decision, $\hxt$, and the round optimum computed in hindsight, $\mathbf{b}_t^\star \in \argmin_{\xx \in \left\{0,1 \right\}^n} f_t \left( \xx \right)$. In this case, the dynamic regret takes the following form:
\[
\Reg_\tau = \sum_{t=1}^\tau f_t(\hxt) - f_t(\mathbf{b}_t^\star),
\]
where $\hxt$ is the decision at $t$. Because we use an additional randomization step to deal with the binary nature of the decision variable, we bound the expectation of the dynamic regret. 
We seek an algorithm with a sublinear bound~\cite{shalev2012online,hazan2016introduction}, which implies that $\E\left[\Reg_\tau\right]/\tau$, the time-averaged regret, diminishes as the time horizon increases, and thus so does the average decision error. Finally, it is common to bound the dynamic regret as a function of the cumulative variation in the relaxed round optima, $V_\tau$~\cite{zinkevich2003online,hall2015online}. This term characterizes how much the sequence of relaxed optima changes through time. Let $V_\tau = \sum_{t=1}^{\tau-1} \left\|\xx_{t+1}^\star - \xx_t^\star \right\|_2$ where $\xx_{t}^\star \in \argmin_{\xx \in \left[0,1 \right]^n} f_t \left( \xx \right)$.

We use the \texttt{bOGD} update to solve~\eqref{eq:nonconvex}, given by
\begin{align}
\xx_{t+1} &= \argmin_{\xx \in \left[0,1 \right]^n} \eta \nabla f_{t}(\xx_{t})^\top \xx + \frac{1}{2}\left\| \xx - \xx_{t} \right\|^2_2 + \eta \lambda \left\|\xx\right\|_1 \label{eq:update_1}\\
\hat{\xx}_{t+1} &= \mathcal{R}\left( \xx_{t+1} \right),\label{eq:update_2}
\end{align}
with $\xx_{1} \in [0,1]^n$ and $\hat{\xx}_{1} = \mathcal{R}\left( \xx_{1} \right)$, and where $\lambda \geq 0$ controls the optional $\ell_1$-regularizer, $\eta >  0$ is the descent step size to be set shortly, and $\mathcal{R}:\left[ 0 , 1 \right]^n \mapsto \left\{0,1 \right\}^n$ is a randomization function. The function $\mathcal{R}$ returns a vector made only of zeros or ones where each entry $i$ is sampled from a Bernoulli random variable with probability $p=\xt(i)$. We note that $\E\left[ \hxt \right] \equiv \E\left[ \mathcal{R}\left(\xt \right) \right] = \xt$. 

\texttt{bOGD} is the first regret-bounded algorithm for OCO with binary constraints. It works as follows. 
In~\eqref{eq:update_1}, the auxiliary, continuous variable $\xx_t$ is updated using dynamic mirror descent with an $\ell_1$-regularizer and the $\ell_2$-norm-based Bregman divergence~\cite{hall2015online}. In~\eqref{eq:update_2}, the randomization function converts the auxiliary variable to the binary decision vector $\hat{\xx}_{t+1}$. The \texttt{bOGD} update combines~\eqref{eq:update_1} and~\eqref{eq:update_2} to first perform a projection gradient descent step on the relaxed, continuous problem at time $t$, and then to obtain a binary decision using randomization. A full description of \texttt{bOGD} is presented in Algorithm~\ref{alg:bogd_alg}.

\begin{algorithm}[!t]
\begin{algorithmic}[1]
\STATEx \textbf{Parameters:} $T$, $a$, $\lambda$, and $\tau$.
\STATEx \textbf{Initialization:} Set $\xx_1 \in \left[0,1 \right]^n$, $\eta = \frac{a}{\sqrt{T}}$.
\medskip

\FOR{$t = 1,2, \ldots, \tau$}
\STATE Implement binary decision $\hxt = \mathcal{R}\left( \xt \right)$.

\STATE Observe the loss function $f_t(\hxt)$.

\STATE Compute the auxiliary variable $\xx_{t+1}$:
\begin{align*}
\xx_{t+1} = \argmin_{\xx \in \left[0,1 \right]^n} \eta \nabla f_{t}(\xx_{t})^\top \xx + \frac{1}{2}\left\| \xx - \xx_{t} \right\|^2_2 + \eta \lambda \left\|\xx\right\|_1.
\end{align*}
\ENDFOR
\end{algorithmic}
\caption{Binary online gradient descent (\texttt{bOGD})}
\label{alg:bogd_alg}
\end{algorithm}

\section{Regret Analysis}
\label{sec:regret_analysis}
We now present two intermediate results that will be used to derive \texttt{bOGD}'s dynamic regret bounds. We first re-express the expected regret of \texttt{bOGD} in terms of a tracking and a rounding component.
\begin{lemma} The expected regret of \texttt{bOGD} is bounded from above by
\begin{equation}
\E\left[\Reg_\tau \right] \leq \Reg_\tau\left(\text{relaxed}\right) + \sum_{t=1}^\tau \E\left[\left|f_t(\hxt) - f_t(\xt) \right| \right], \label{eq:sum_of}
\end{equation}
where $\Reg_\tau\left(\text{relaxed}\right)= \sum_{t=1}^\tau f_t(\xt) - f_t(\mathbf{x}_t^\star).$
\label{lem:regret_is_sum}
\end{lemma}
The proof of this lemma is in Appendix~\ref{app:proof_of_sum}. Lemma~\ref{lem:regret_is_sum} shows that the expected \texttt{bOGD} regret is upper bounded by the sum of the relaxed, continuous problem regret (first right-hand side term of~\eqref{eq:sum_of}) and the cumulative rounding error due to the randomization step (second right-hand side term of~\eqref{eq:sum_of}). We now bound the expected rounding error, i.e., the expected difference between the loss associated with the randomized binary and relaxed decisions.
\begin{lemma}
Let the binary decision variable $\hxt$ be computed using \texttt{bOGD}. Then
\[
\E\left[ \left|f_t\left( \hxt \right) - f_t\left( \xt \right) \right| \right] \leq \frac{L_2 \sqrt{n}}{2}.
\]
\label{lem:err_rand_2}
\end{lemma}

The proof of Lemma~\ref{lem:err_rand_2} is provided in Appendix~\ref{app:proof_lemma_rand_2}. Next, we present our regret bounds for \texttt{bOGD}. We provide a general bound in Theorem~\ref{thm:bv_ogd_bound} and then give a sublinear, finite-time regret bound in Corollary~\ref{cor:finite_time}.

\begin{theorem}[Regret bound for \texttt{bOGD}]
Let $\eta = \frac{a}{\sqrt{\tau}}$, $a > 0$. The expected regret of \texttt{bOGD} is upper bounded by
\[
\begin{aligned}
\E\left[\Reg_\tau \right] &\leq \left(\frac{n}{2a}  + \frac{a\left(L_2\right)^2}{2}\right)\sqrt{\tau} + 2n \sqrt{\tau} V_\tau + \frac{L_2\sqrt{n}\tau}{2},
\end{aligned}
\]
for any time horizon $\tau$.
\label{thm:bv_ogd_bound}
\end{theorem}

The proof of Theorem~\ref{thm:bv_ogd_bound} is provided in Appendix~\ref{app:proof_bound}. The regret bound of \texttt{bOGD} is based on Lemmas~\ref{lem:regret_is_sum} and~\ref{lem:err_rand_2}. Using Theorem~\ref{thm:bv_ogd_bound}, we can bound the expected regret with the sum of the rounding error and the error due to tracking relaxed optima. The tracking error is equivalent to the error made by the dynamic mirror descent~\cite{hall2015online}. The relaxed optimum tracking regret is sublinear if $V_\tau < O(\sqrt{\tau})$. Lemma~\ref{lem:err_rand_2} quantifies the rounding error in each round. It grows linearly in $\tau$ because rounding error will inevitably occur in almost all rounds.

Using Lemmas~\ref{lem:regret_is_sum} and~\ref{lem:err_rand_2}, we can provide a sublinear regret bound for \texttt{bOGD} when it is used over the constrained time horizon $T \leq \tau$. The constrained time horizon $T$ is based on the Lipschitz continuity modulus with respect to the $1$-norm and the step size scaling factor, $a$. It can therefore be computed before implementation.
\begin{lemma}
Suppose $T \leq \left(aL_1\right)^{2}$ and $a > 0$, and set $\eta = \frac{a}{\sqrt{T}}$. The expected regret of \texttt{bOGD} is upper bounded by
\[
\begin{aligned}
\E\left[\Reg_T \right] &\leq \left(\frac{n}{2a}  + \frac{a\left(L_2\right)^2}{2} + \frac{a L_1 L_2 \sqrt{n}}{2}\right)\sqrt{T} + 2nL_1 V_T.
\end{aligned}
\]
\label{lem:bv_ogd_bound}
\end{lemma}

The proof of Lemma~\ref{lem:bv_ogd_bound} is presented in Appendix~\ref{app:proof_bv_ogd_bound}. Lemma~\ref{lem:bv_ogd_bound} means that if $V_T$ is sublinear, the expected worst-case performance bound of the decision maker will improve over time for $t\leq T$. We use Lemma~\ref{lem:bv_ogd_bound} to obtain the following regret bound for any finite time horizon, $\tau < \infty$. This is similar to the regret analysis of~\cite{shahrampour2017distributed}.
\begin{corollary}[Finite-time regret bound for \texttt{bOGD}]
Let $T \leq \left(aL_1\right)^{2}$ and $a > 0$. Using \texttt{bOGD} with $\eta = \frac{a}{\sqrt{T}}$, the expected cumulative regret over the finite time horizon, $\tau  < \infty,$ is bounded from above by:
\[
\E\left[\Reg_\tau \right] \leq \left(\frac{n}{2a} + \frac{3aL^2_2}{2} + \frac{a L_1 L_2 \sqrt{n}}{2} \right) \tau^\epsilon + 2nL_1 V_\tau,
\]
where $0 < \epsilon < 1$. Consequently, $\E\left[ \Reg_\tau \right]$ is sublinear if $V_\tau < O (\tau)$.
\label{cor:finite_time}
\end{corollary}

\begin{proof}
Suppose that the algorithm is restarted after $T \leq \left(aL\right)^{2}$ rounds. Let $m \in \mathbb{N}$ be the number of times the algorithm has been restarted. Let $N = \tau/T$ be the number of times the algorithm is required to be restarted throughout the time horizon $\tau$. We assume that $N \in \mathbb{N}$, rounding up $\tau/T$ if it is not integer.

Let $\Reg_{t_0:t_1} = \sum_{t=t_0}^{t_1} f_t(\hxt) - f_t(\mathbf{b}_t^\star)$ be the cumulative regret from round $t_0$ to $t_1$.
By Lemma~\ref{lem:bv_ogd_bound}, we have,
\[
\begin{aligned}
\E\left[\Reg_{{(m-1)T+1:mT}} \right] &\leq c_1 \sqrt{{mT - (m-1)T}} \\
&\quad+ c_2 V_{{(m-1)T+1:mT}},
\end{aligned}
\]
where $c_1 = \frac{n}{2a}  + \frac{a\left(L_2\right)^2}{2} + \frac{a L_1 L_2 \sqrt{n}}{2}$, $c_2=2nL_1$, and $V_{t_0:t_1} = \sum_{t=t_0}^{{t_1-1}} \left\| \xx_{t+1}^\star - \xx_{t}^\star \right\|$. Then, the cumulative regret is 
\begin{align}
\E\left[\Reg_\tau\right] &= \sum_{m=1}^N \E\left[\Reg_{{(m-1)T+1:mT}} \right] \nonumber\\
&\leq \sum_{m=1}^N c_1 T^{\frac{1}{2}} + c_2 \sum_{m=1}^N V_{{(m-1)T+1:mT}}\nonumber\\
&= c_1 N T^{\frac{1}{2}}  + c_2 V_\tau. \label{eq:reg_tau}
\end{align}
By definition, $\tau = N T$, thus there exists $0 < \epsilon < 1$ {given by $\epsilon = \frac{\ln\left(N \sqrt{T} \right)}{\ln\tau}$} such that $\tau^\epsilon =  N T^\frac{1}{2}$. We rewrite~\eqref{eq:reg_tau} as
\[
\E\left[\Reg_\tau \right] \leq c_1  \tau^\epsilon + c_2 V_\tau,
\]
which completes the proof.
\end{proof}

Corollary~\ref{cor:finite_time} shows that for any finite time horizon and $V_\tau < O\left( \tau \right)$, the regret is sublinear and the time-averaged regret ($\E\left[R_\tau\right]/\tau$) decreases as the time horizon increases. Due to our finite time analysis, we can accommodate $V_\tau$ that is almost linear in $T$. We expect that similar results would hold for standard OCO algorithms~\cite{zinkevich2003online,hall2015online} if they were to be analyzed similarly.

Corollary~\ref{cor:finite_time}'s bound is looser than prevalent dynamic OCO algorithms~\cite{zinkevich2003online,hall2015online,mokhtari2016online} because of the $\tau^{\epsilon}$ term. If, for example, $\tau = 10^4$ and $T=100$, the regret is bounded by $O\left( \tau^{3/4} \right) + O \left(V_\tau \right)$. This increased regret bound order is explained in part by the randomization step used to deal with the binary constraints. The randomization adds an extra term to the regret of standard continuous-variable online gradient descent-based algorithms, as shown in~Lemma~\ref{lem:regret_is_sum}. This term increases the order of the regret bound when moving from a constrained time horizon to any finite time horizon.

\section{Application to demand response}
\label{sec:application}

We use \texttt{bOGD} for setpoint tracking with thermostatically controlled loads. Our formulation includes three types of constraints on the TCLs: (i) temperature preference constraints, (ii) manual override by the load user and (iii) lockout. The lockout imposes a $M$-minute downtime period when an air conditioner is off regardless of other constraints. Lockout is standard on most commercial air conditioning units and prevents their compressor from being turned off and on repeatedly~\cite{zhang2013aggregated}. Otherwise, operating the unit in such fashion can degrade the unit’s efficiency and life span due to compressor failure~\cite{zhang2013aggregated,vrettos2016fast}, and for these reasons, it is avoided. These constraints are modeled with round-dependent parameters of the loss function which we introduce next. Their only impact to the regret bound of \texttt{bOGD} is with respect to the Lipschitz constants $L_1$ and $L_2$ of the loss function, and cumulative variation $V_T$ of the OCO problem. We present next the setpoint tracking setup and then evaluate its performance in numerical simulations. 

\subsection{Setup}
\label{sec:background} 

We consider $n$ TCLs. The load parameters are:
\begin{itemize}
	\item $\mathbf{r} \in \RR_+^n$, where $\mathbf{r}(i)$ is the thermal resistance of load $i$;
	\item $\mathbf{c} \in \RR_+^n$, where $\mathbf{c}(i)$ is the thermal capacitance of load $i$;
	\item $\bm{\theta}_\text{d} \in \RR^n$, where $\bm{\theta}_\text{d}(i)$ is load $i$'s  desired temperature;
	\item $\overline{\bm{\theta}}, \underline{\bm{\theta}} \in \RR^n$, where $\overline{\bm{\theta}}(i)$ and $\underline{\bm{\theta}}(i)$ are respectively the maximum and minimum temperature of load $i$'s temperature deadband.
\end{itemize}
The load parameters are assumed to be known. The following online parameters are observed at the end of each round:
\begin{itemize}
	\item $s_t \in \RR$, the power setpoint to track;
	\item $\tpt \in \RR_+^n$, where $\tpt(i)$ is the power consumption at time $t$ of load $i$'s cooling unit when it is not available for demand response because it is constrained by its temperature, lockout, or manual override.
	\item $\pt \in \RR_+^n$, where $\pt(i)$ represents the power consumption of load $i$ and time $t$ when it is controllable, i.e., when it is not constrained by the deadband, lockout, or manual override.
	\item $\tt \in \RR^n$, where $\tt(i)$ is the indoor temperature at time $t$ of load $i$. Let $\left\langle \bm{\theta} \right\rangle_{t} \in \RR^n$ be the vector of temperatures averaged over rounds $1$ to $t$;
	\item $\theta_t^\text{ambient} \in \RR$ is the ambient (outdoor) temperature. The ambient temperature is the same for all loads. This assumption simplifies the computation but can be relaxed;
	\item $\uu_t \in \{0,1\}^n$ is the cooling unit override at time $t$ of load $i$. When $\uu_t(i)=1$, the local cooling unit overrides the demand response control because the temperature is above the deadband. It also represents rounds when the user manually controls the cooling unit, which we model as uncertain.
\end{itemize}
We note that the online parameters represent observations or measurements. 

Lastly, let $\hxt(i)=1$ and $\hxt(i)=0$ be, respectively, the signal to turn on and off load $i$'s cooling unit.

\setlength{\textfloatsep}{2pt}
\begin{algorithm}[!t]
\begin{algorithmic}[1]
\STATEx \textbf{Parameters:} $T$, $a$, $\rho$, $\lambda$, $h$
\STATEx \textbf{Initialization:} Set $\xx_1 \in \left[0,1 \right]^n$, $\eta = \frac{a}{\sqrt{T}}$, and $K = \frac{5}{60h}$
\medskip

\FOR{$t = 1,2, \ldots, T$}
\STATE Implement binary decisions $\hxt = \mathcal{R}\left( \xt \right)$.
\STATE Observe all online parameters: $\pt$, $\tpt$ $\uu_t$, $\tt$, $s_t$, $\theta_t^\text{ambient}$.
\STATE Compute the mean temperature $\left\langle \bm{\theta} \right\rangle_{t-1}$.

\STATE Update the relaxed decision variable $\xx_{t+1}$:
\begin{align*}
\xx_{t+1} = \argmin_{\xx \in \left[0,1 \right]^n} \eta \nabla f_{t}(\xx_{t})^\top \xx + \frac{1}{2}\left\| \xx - \xx_{t} \right\|^2_2 + \eta \lambda \left\|\xx\right\|_1.
\end{align*}
\ENDFOR
\end{algorithmic}
\caption{\texttt{bOGD} for setpoint tracking with TCLs}
\label{alg:controller}
\end{algorithm}

The model and update rule are given in Algorithm~\ref{alg:controller}. We use~\eqref{eq:loss_function} from~\cite{lesage2018setpoint} as the loss function. In~\eqref{eq:loss_function}, $\rho>0$ and $\lambda>0$ are numerical parameters. There are three terms, as explained below.

\begin{figure*}[!t]
\begin{equation}
\begin{aligned}
f_t \left( \hxt \right) &= \left( s_t - \pt^\top \hxt - \mathbf{\tilde p}_t^\top \uu_t\right)^2 + \lambda \left\| \hxt \right\|_1 \\
&\qquad +\frac{\rho}{2} \left\|\frac{t-1}{t}\mtt + \frac{1}{t}\left(\vphantom{\frac{1}{t}}\mathbf{B} \bm{\theta}_{t-1} + (\mathbf{I} - \mathbf{B})\left(\oo\theta_t^\text{ambient}  - \hxt \diag \left(\mathbf{r}\right) \diag \left(\pt\right) \right) \right) - \bm{\theta}_\text{d}\right\|_2^2
 \label{eq:loss_function}
\end{aligned}
\end{equation}
\hrulefill
\end{figure*}
\begin{enumerate}
	\item \emph{Setpoint tracking losses.} This term is used to match the total power consumption of the loads to the regulation signal, $s_t$. The power consumption of the TCLs is split into two components: the controllable part dispatched by the algorithm, $\pt^\top \hxt$, and the uncontrollable part, $\mathbf{\tilde p}_t^\top \uu_t$, set by the loads to meet their constraints. The squared tracking error penalizes large deviations.
	\item \emph{Sparsity regularizer.} This $1$-norm term is used to avoid sending nonzero control signals to loads that are not available at a given round, e.g., if their temperature is out of the deadband or during lockout. The availability of a load $i$ is expressed by $\pt(i) \neq 0$. It is also used to minimize the number of loads dispatched and to avoid sending small signals to the loads. The sparsity regularizer is directly incorporated to the update instead of being taken into account via the gradient. This is done to improve the regularization performance~\cite{duchi2010composite}.
	\item \emph{Mean temperature regularizer.} This $2$-norm term minimizes the long-term impact of demand response on the loads by promoting a small difference between the average and desired temperatures. While the temperatures used in Algorithm~\ref{alg:controller} are based on measurements, we use the following discrete-time approximation to model the impact of turning on the cooling units on the temperature~\cite{callaway2009tapping,mathieu2015arbitraging}:
	\begin{equation}
	\bm{\theta}_{t+1} = \mathbf{B} \bm{\theta}_{t} + (\mathbf{I} - \mathbf{B})\left(\oo\theta_t^\text{ambient}  - \hxt \diag \left(\mathbf{r}\right) \diag \left(\pt\right) \right), \label{eq:temp}
	\end{equation}
	where $\oo$ is the $n$-dimensional vector of ones, $\mathbf{B} = \diag \left(\exp\left( - \frac{h}{\mathbf{r}(i)\mathbf{c}(i)} \right), \text{ for } i =1,2,\ldots, n \right)$, and $h>0$ is the round duration.
\end{enumerate}

The temperature deadband, the lockout constraints, and the manual override are incorporated into the model via $\pt$. The loads measure their power consumption according to their availability and set $\pt$ locally, which is then observed by the algorithm. A load is available if its temperature is within the deadband, not in lockout, and not manually controlled by the user. The logical rules governing $\pt(i)$ are summarized below:
\begin{itemize}
	\item (\emph{deadband}) If $\bm{\theta}_{t-1}(i) < \underline{\bm{\theta}}(i)$, then the load sets $\pt(i) = 0$.
	\item (\emph{deadband}) If $\bm{\theta}_{t-1}(i) > \overline{\bm{\theta}}(i)$, then the load sets $\pt(i) = 0$. The cooling unit is activated using override control, i.e., $\uu_t(i) = 1$ and $\tpt(i) \neq 0$.
	\item (\emph{manual override}) If the load manually uses its cooling unit regardless of deadband constraints and demand response instructions, $\uu_t(i) = 1$, $\pt(i) = 0$, and $\tpt(i) \neq 0$.
	\item (\emph{lockout}) If $\mathbf{p}_{k-1}(i) \mathbf{\hat{x}}_{k-1}(i) + \tilde{\mathbf{p}}_{k-1}(i) \uu_{k-1}(i) > 0$ and $\mathbf{p}_{k}(i) \mathbf{\hat{x}}_{k}(i) + \tilde{\mathbf{p}}_{k}(i) \uu_{k}(i) = 0$ for any $k \in \left\{t-1,t-2,\ldots, t-K\right\}$ where $K$ is the number of rounds representing a duration of $M$ minutes, then $\pt(i) = 0$ and $\uu_t(i) = 0$. In other words, if the cooling unit was shut down in the past $M$ minutes, then it is unavailable for control.
	\item (\emph{available load}) Otherwise, load $i$ is available and $\pt(i)\neq 0$ and $\uu_t(i) = 0$.
\end{itemize}
Load $i$ only implements the demand response controls if it is available. Therefore, it only has its air conditioner turned on by the aggregator if $\pt(i) \neq 0$. We assume that the parsing is done locally by the load’s back-up controller. If $\pt(i) = 0$, the demand response control is ignored by the load, but it can be turned on by the override control, e.g., to keep the temperature within the deadband. Note that when in lockout, we have $\pt(i) = 0$ and $\mathbf{u}(i) = 0$, and the air conditioner is always off.

Lastly, a switching cost could be added to~\eqref{eq:loss_function} to penalize frequent changes in controls. By reducing these changes, the load should be less frequently in lockout, and the aggregator would benefit from increased load availability. This is a topic for future work.

\subsection{Numerical example}
\label{sec:example}

We consider an aggregation of $n=1000$ TCLs. We assume that the loads have access to two-way communication infrastructure. They can receive cooling instructions and can report their power consumption and temperature to and from the demand response aggregator. For example, the different sensors, thermostat and cooling unit's controller, can all be connected using different local communication protocols, e.g. Wi-Fi or ZigBee~\cite{gungor2012survey}, to a hub which is itself internet enabled. Information to and from the demand response aggregator can then be communicated via the Internet.

All load parameters are sampled uniformly from~\cite[Table I]{mathieu2015arbitraging}. We consider a round duration of $h=1$ minute and set the lockout time $M$ to be $5$ minutes. Let $s_t = s_{0} + w_t$ be the random and intermittent regulation signals where $w_t$ is constant for $5$ rounds and then sampled according to $w_t \sim \mathrm{N}(0,300)$.  We set $s_{0}=2400$ kW so that loads can balance maintaining their temperature within their deadband and tracking the setpoints. The initial decision $\xx_1(i)$ is set randomly to $0$ or $1$ with a probability of $0.5$ for all $i=1,2,\ldots,n$.

The ambient temperature for the short time horizon is $\theta_t^\text{ambient} = \theta_0^\text{ambient} + 0.25 \sin \left( t \pi / T\right)$ where $\theta_0^\text{ambient} = 34^\circ$C. We let the initial temperature of the TCLs be their desired ones, $\bm{\theta}_\text{d}$. The load temperature evolves as~\eqref{eq:temp} plus zero-mean Gaussian noise with variance $0.025$. The regularizer parameters are set to $\sigma = 500$ and $\lambda = 250$, and the step size parameter to $a=4 \times 10^{-4}$. For the numerical calculations, we use the parser \texttt{CVXPY}~\cite{cvxpy} with the solver Gurobi~\cite{gurobi}.

Figure~\ref{fig:dr} shows the power consumption of \texttt{bOGD}; of \texttt{bOGD} without randomization; and of the optimal dispatch, all of which include the uncontrollable part of the TCLs power consumption; the setpoint; and the uncontrollable power consumption presented individually. The uncontrollable power consumption, $\tpt^\top \uu_t$, represents all decisions taken locally by the loads to ensure that their temperatures stay within the deadband. The relative difference between the power consumption using randomized decisions and the relaxed, continuous decisions is, on average, $1.30\%$. The average relative tracking error of the randomized and relaxed decisions are $6.51\%$ and $6.46\%$, respectively. We note the higher tracking error at the start of the simulation. This is due to the limited number of temperature measurements used by the mean-temperature regularizer in the first rounds. We present the setpoint tracking root mean square error (RMSE) for \texttt{bOGD} and its counterpart without randomization (relaxed decisions) in Table~\ref{tab:rmse}. The second column of Table~\ref{tab:rmse} also shows the relative RMSE with respect to the average of the signal, $\langle s_t \rangle =2412.09$ kW. The RMSE corroborates the small difference in performance between \texttt{bOGD} with and without randomization (relaxed) when used in demand response. The method requires, on average, $90.2$ milliseconds on a $2.4$ GHz Intel Core i$5$ laptop computer to evaluate the demand response decisions for $n=1000$ loads. This illustrates the algorithm's high scalability and the potential for real-time implementation.

\begin{figure}[tb]
  \centering
  \includegraphics[width=1\columnwidth]{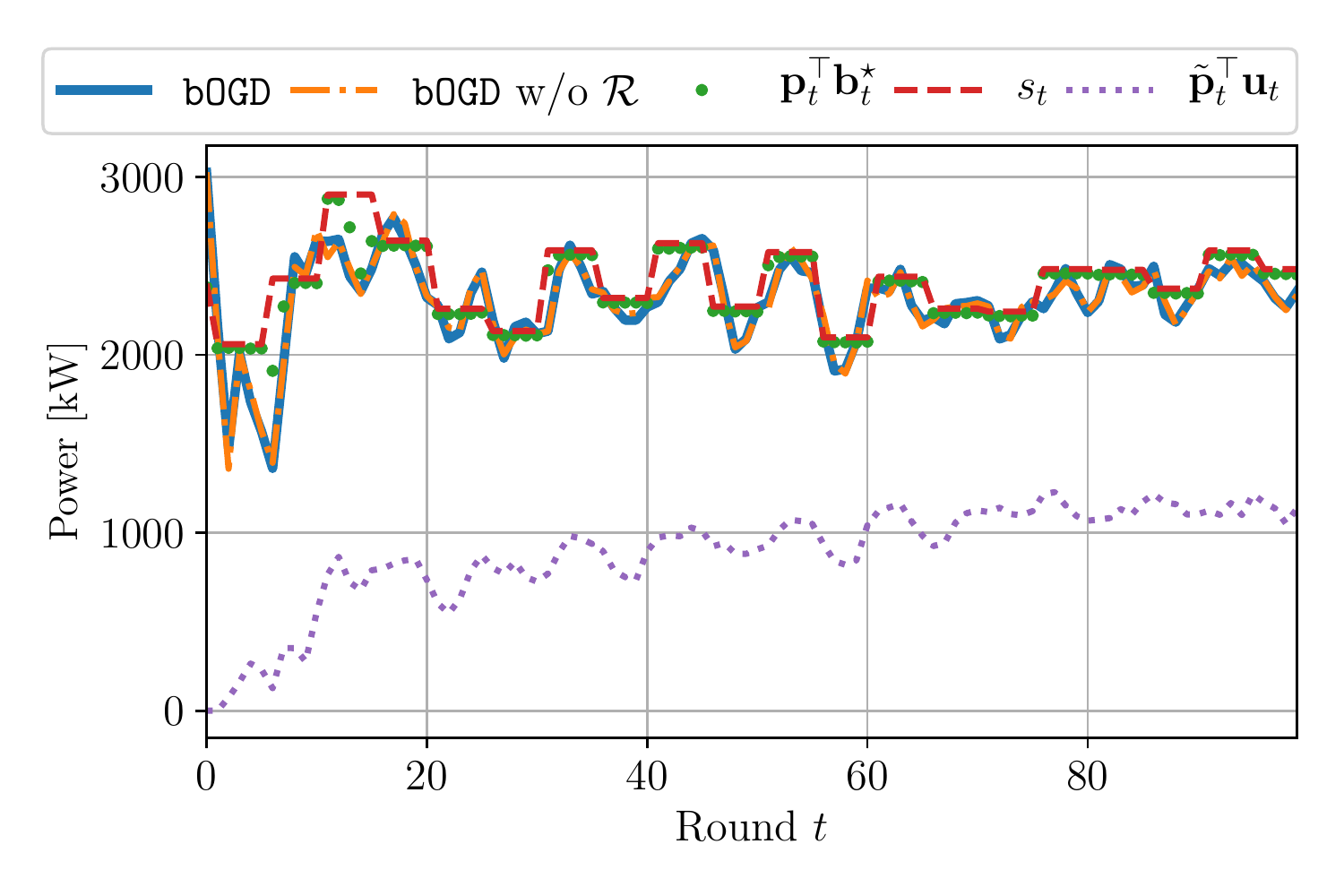}
  \vspace{-1cm}
  \caption{Demand response performance of \texttt{bOGD} with TCLs}
  \label{fig:dr}
\end{figure}

\begin{table}[tb]
  \caption{Setpoint tracking root mean square error for \texttt{bOGD}}
  \renewcommand{\arraystretch}{1.3}
  \centering

  \begin{tabular}{ccc}
  \hline

  \hline
  \textbf{Algorithm} & \textbf{RMSE [kW]}  & \textbf{relative RMSE} \\
  \hline
  \texttt{bOGD}   & $226.88$ & $9.41\%$\\
  \texttt{bOGD} w/o $\mathcal{R}$   & $229.17$ & $9.50\%$\\
  \hline

  \hline
  \end{tabular}
  \label{tab:rmse}
\end{table}

Figure~\ref{fig:temperature_tcls} shows two temperature profiles observed during the simulation. In Figure~\ref{fig:1000}, the TCL is dispatched on several occasions, including once over several rounds. Near the end of the simulation, the TCL is constrained to the upper limit of its deadband. Because of the randomization step, a load may become stuck at a high temperature, thus preventing it from providing sustained flexibility throughout the time horizon. Figure~\ref{fig:100} presents a load that, when lockout and deadband constraints permit, is frequently dispatched to follow the setpoint, sometimes providing an increase in power consumption continuously over several rounds.

\begin{figure*}[tb]
\centering
\subfloat[Load $i=361$]{\includegraphics[width=1\columnwidth]{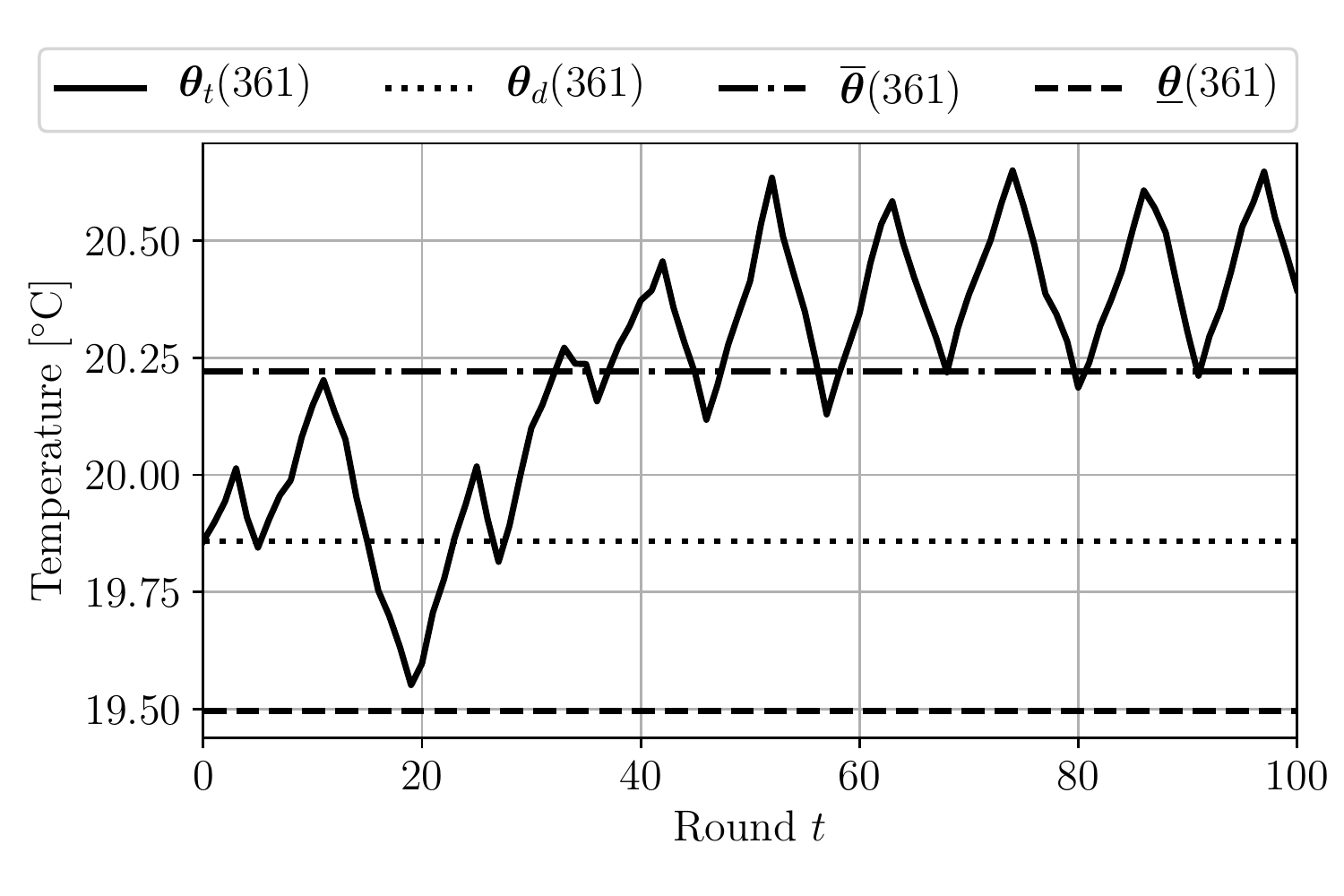}
\label{fig:1000}
}
\subfloat[Load $i=369$]{\includegraphics[width=1\columnwidth]{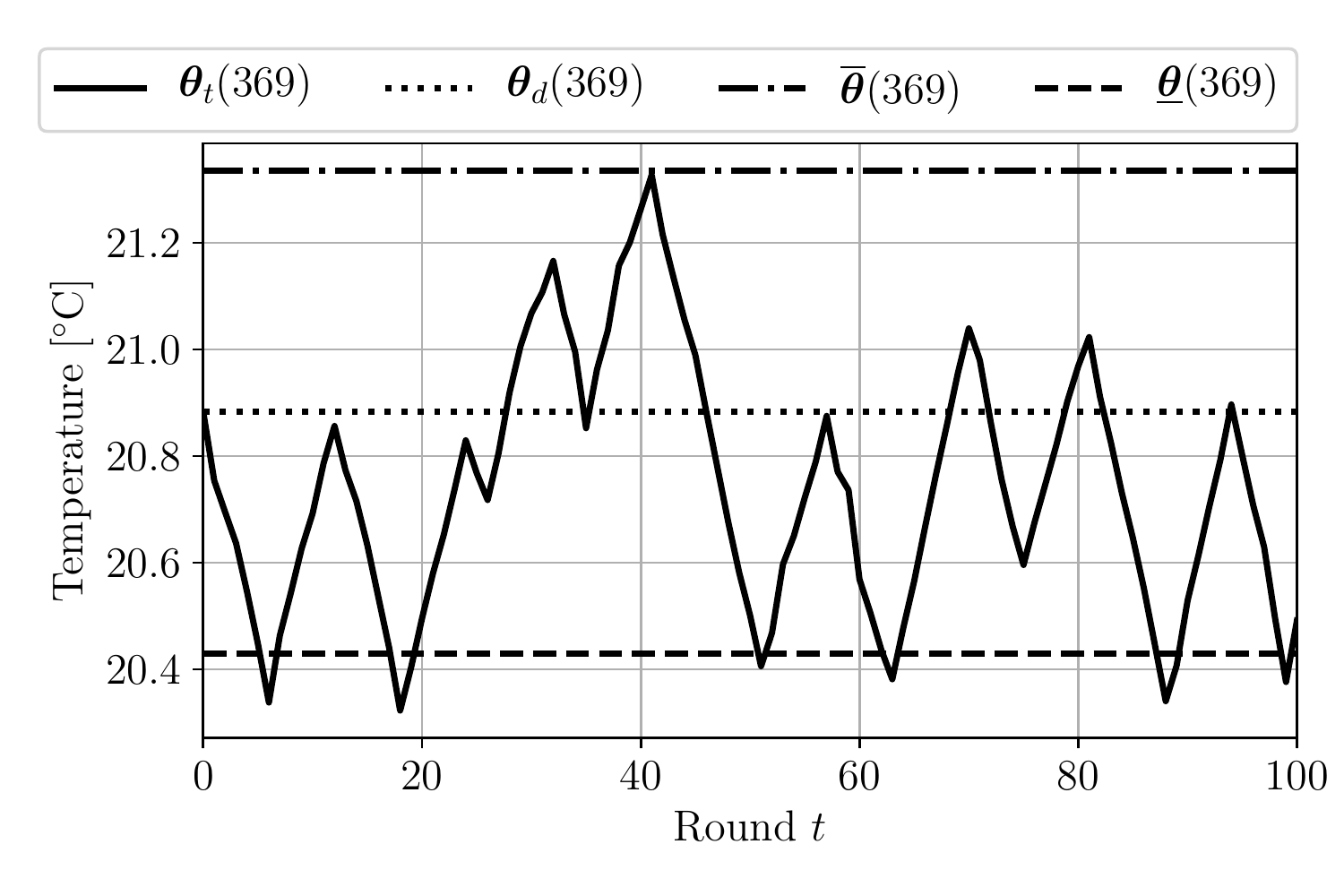}
\label{fig:100}
}
\vspace{-0.1cm}
\caption{Temperature profile of two selected TCLs during demand response}
\label{fig:temperature_tcls}
\end{figure*}

Figure~\ref{fig:regret} shows the regret for \texttt{bOGD} with and without randomization (relaxed decisions), averaged over $100$ randomization steps ($\langle$\texttt{bOGD}$\rangle_{100}$), and its finite-time regret bound. The averaged regret approximates the expected regret. Figure~\ref{fig:regret} shows that in the current setting the regret (i) outperforms the bound provided in Corollary~\ref{cor:finite_time}, (ii) is sublinear, thus showing that the decisions made with \texttt{bOGD} are approaching the round optimum at $t$ increases, and (iii) is similar to the regret without randomization and averaged over several simulations.

\begin{figure}[tb]
  \centering
  \includegraphics[width=1\columnwidth]{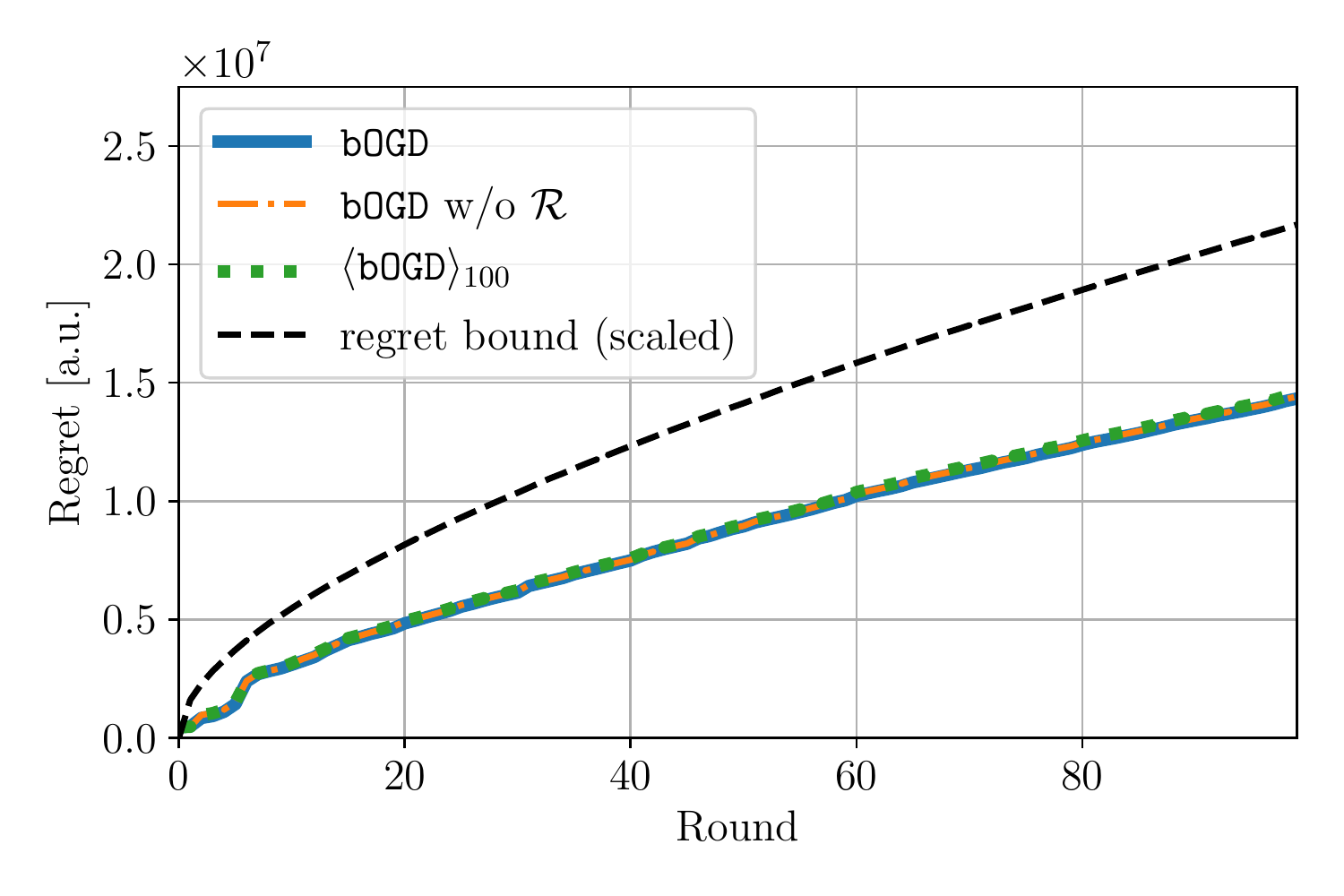}
  \vspace{-1cm}
  \caption{Regret comparison of \texttt{bOGD}}
  \label{fig:regret}
\end{figure}	

\section{Conclusion}
\label{sec:conclu}
In this paper, motivated by real-world demand response constraints, we incorporate binary constraints into OCO. We show that the expected dynamic regret of \texttt{bOGD}: (i) is upper bounded by the sum of the cumulative relaxed optimum tracking and the rounding error when the time horizon is unconstrained and (ii) has a sublinear in time and linear in cumulative variation in round optima bound when constrained to a finite time horizon. We use \texttt{bOGD} for real-time demand response with TCLs. We model the discrete on/off settings of the cooling units and unavailability due to lockout, temperature deadbands, and manual override.

Numerical simulations show that the randomization step yields a $1.30\%$ deviation, on average, from the relaxed decisions. The relative tracking error is $6.51\%$ and $6.46\%$, respectively, for the binary and relaxed algorithms. The algorithm is easy to implement and does not need extensive monitoring, data, or models. On average, less than $0.1$ second is required to compute decisions for $1000$ loads.

Future directions for this work include obtaining an asymptotic and tighter regret bound for \texttt{bOGD} and modeling general, time-invariant equality and inequality constraints, for example, using randomized rounding techniques.

\section*{Acknowledgment}
The authors thank Dr. J. Yuan for the helpful discussion.

\appendix

\subsection{Proof of Lemma~\ref{lem:regret_is_sum}}
\label{app:proof_of_sum}

The expected regret of \texttt{bOGD} is
\[
\E\left[\Reg_\tau \right] = \E\left[\sum_{t=1}^\tau f_t(\hxt) - f_t(\mathbf{b}_t^\star) \right],
\]
where we recall that $\mathbf{b}_t^\star \in \argmin_{\xx_t \in \left\{0,1\right\}^n} f_t(\xt)$ is a binary-integer optimum. We can bound from above the right-hand side by
\[
\E\left[\Reg_\tau \right] \leq \E\left[\sum_{t=1}^\tau f_t(\hxt) - f_t(\xt^\star) \right],
\]
because $\xst$ is a relaxed optimum and $f_t(\xst) \leq f_t(\mathbf{b}_t^\star)$. Re-arranging the terms, we obtain
\begin{align*}
\E\left[\Reg_\tau \right] &= \sum_{t=1}^T \E\left[ f_t(\hxt) + f_t(\xt) - f_t(\xt) - f_t(\xt^\star) \right]\\
&\leq \sum_{t=1}^\tau \E\left[\left|f_t(\hxt) - f_t(\xt) \right| \right] + f_t(\xt)- f_t(\xt^\star)\\
&= \Reg_\tau\left(\text{relaxed}\right) + \sum_{t=1}^\tau \E\left[\left|f_t(\hxt) - f_t(\xt) \right| \right],
\end{align*}
which completes the proof. $\qedproof$

\subsection{Proof of Lemma~\ref{lem:err_rand_2}}
\label{app:proof_lemma_rand_2}
The Lipschitz continuity of $f_t$ with respect to $\left\| \cdot \right\|_2$ yields
\[
\E\left[ \left|f_t\left( \hxt \right) - f_t\left( \xt \right) \right| \right] \leq \E \left[ L_2 \left\| \hxt - \xt \right\|_2\right].
\] 
We re-arrange the right-hand side to obtain
\begin{align}
\E\left[ \left|f_t\left( \hxt \right) - f_t\left( \xt \right) \right| \right] &\leq L_2 \E \left[ \sqrt{\sum_{i=1}^n\left( \hxt(i) - \xt(i) \right)^2}\right] \nonumber\\
&\leq L_2 \sqrt{  \sum_{i=1}^n \E \left[ \left( \hxt(i) - \xt(i) \right)^2 \right]}, \label{eq:here_exp}
\end{align}
where we have used Jensen's inequality for concave function to obtain the second inequality. By definition of the randomization function, we have $\E \left[ \hxt \right] = \xt$, and thus,~\eqref{eq:here_exp} can be re-expressed as:
\[
\E\left[ \left|f_t\left( \hxt \right) - f_t\left( \xt \right) \right| \right]\leq L_2 \sqrt{  \sum_{i=1}^n \E \left[ \left( \hxt(i) - \E \left[ \hxt (i)\right] \right)^2 \right]}.
\]
The expectation now represents the variance. Re-expressing the variance in terms of the first and second moment leads to
\[
\E\left[ \left|f_t\left( \hxt \right) - f_t\left( \xt \right) \right| \right]\leq L_2 \sqrt{  \sum_{i=1}^n \E \left[ \hxt(i)^2 \right] - \E \left[ \hxt(i) \right]^2  }.
\]
We have $\E \left[ \hxt(i)^2 \right] = \xt(i)$ and $\E \left[ \hxt(i) \right]^2 = \xt(i)^2$ which follows from the definition of $\mathcal{R}$. Hence, we obtain
\begin{align*}
\E\left[ \left|f_t\left( \hxt \right) - f_t\left( \xt \right) \right| \right] &\leq L_2 \sqrt{  \sum_{i=1}^n \xt(i) - \xt(i)^2  }\\
&\leq L_2 \sqrt{ \frac{n}{4}}
\end{align*}
and we have completed the proof. $\qedproof$

\subsection{Proof of Theorem~\ref{thm:bv_ogd_bound}}
\label{app:proof_bound}
We use Lemma~\ref{lem:regret_is_sum} and express the regret as:
\begin{equation}
\E\left[\Reg_\tau \right] \leq \Reg_\tau\left(\text{relaxed}\right) + \sum_{t=1}^\tau \E\left[\left|f_t(\hxt) - f_t(\xt) \right| \right] \label{eq:sum_of_bis}
\end{equation}
where $\Reg_\tau\left(\text{relaxed}\right)$ denotes the regret of the relaxed, continuous problem solved using dynamic mirror descent~\cite{hall2015online} with $\bm{\Phi}_t$ set to the identity matrix at all $t$, $\psi\left(\xx \right) = \frac{1}{2}\left\| \xx\right\|_2^2$, and $r\left( \xx\right) = \lambda \left\| \xx \right\|_1$. By~\cite[Theorem 2]{hall2015online}, we upper bound the $\Reg_\tau\left(\text{relaxed}\right)$ as
\begin{equation*}
\Reg_\tau\left(\text{relaxed}\right) \leq \frac{D_{\max}}{\eta} + \frac{2L_\psi}{\eta}V_\tau + \frac{\left(L_2\right)^2\eta \tau}{2 \sigma},
\end{equation*}
where $D_{\max} = \frac{n}{2} \geq \frac{1}{2}\left\| \xx - \yy \right\|_2^2$ (the Bregman divergence with respect to $\psi$) for all $\xx$, $\yy \in \XX$, $L_\psi = n$ is a Lipschitz constant for $\frac{1}{2}\left\| \xx - \yy \right\|_2^2$ and $\sigma = 1$, the strong-convexity modulus of $\psi$. We then have
\begin{equation}
\Reg_\tau\left(\text{relaxed}\right) \leq \frac{n}{2\eta} + \frac{2n}{\eta}V_\tau + \frac{\left(L_2\right)^2\eta \tau}{2}. \label{eq:regret_ogd_r1}
\end{equation}
We upper bound~\eqref{eq:sum_of_bis} using~\eqref{eq:regret_ogd_r1} and obtain
\begin{align*}
\E\left[\Reg_\tau \right] &\leq \frac{n}{2\eta} + \frac{2n}{\eta}V_\tau + \frac{\left(L_2\right)^2\eta \tau}{2} \\
&\qquad+ \sum_{t=1}^\tau \E\left[\left|f_t(\hxt) - f_t(\xt) \right| \right]
\end{align*}
Using Lemma~\ref{lem:err_rand_2} gives
\begin{align}
\E\left[\Reg_\tau \right] &\leq \frac{n}{2\eta} + \frac{2n}{\eta}V_\tau + \frac{\left(L_2\right)^2\eta \tau}{2} + \frac{L_2\sqrt{n}\tau}{2} \label{eq:starting_from}
\end{align}
Letting $\eta =\frac{a}{\sqrt{\tau}}$ completes the proof. $\qedproof$

\subsection{Proof of Lemma~\ref{lem:bv_ogd_bound}}
\label{app:proof_bv_ogd_bound}

We start this proof from~\eqref{eq:starting_from} of Appendix~\ref{app:proof_bound}. We have
\begin{align*}
\E\left[\Reg_T \right] &\leq \frac{n}{2\eta} + \frac{2n}{\eta}V_T + \frac{\left(L_2\right)^2\eta T}{2} + \frac{ L_2 \sqrt{n} T}{2}\\
&= \frac{n}{2\eta} + \frac{2nL_1}{\eta L_1}V_T + \frac{\left(L_2\right)^2\eta T}{2} + \frac{L_2 \sqrt{n}T}{2},
\end{align*}
where we have multiplied the second term by $L_1/L_1$. By assumption, $\eta L_1 \geq 1$ and hence
\[
\E\left[\Reg_T \right] \leq \frac{n}{2\eta} + 2nL_1 V_T + \frac{\left(L_2\right)^2\eta T}{2} + \frac{\eta L_1 L_2 T \sqrt{n}}{2}.
\]
Setting $\eta =\frac{a}{\sqrt{T}}$ completes the proof. $\qedproof$


\end{document}